\documentclass{article}
\usepackage{amsmath,amsthm,amssymb,amsfonts}
\usepackage[toc,page]{appendix}

\newtheorem{theorem}{Theorem}[section]
\newtheorem{lemma}[theorem]{Lemma}

\numberwithin{equation}{section}

\newcommand{\eps}       {{\varepsilon}}

\newcommand{\R}         {\mathbb{R}}
\newcommand{\B}         {\mathcal{B}}

\newcommand{\X}        {X}
\newcommand{\x}        {x}

\newcommand{\y}        {y}

\newcommand{\RR}         {\mathcal{R}  }

\newcommand{\CVD}{\hfill $\rule{2.6mm}{2.6mm}$} 

\begin{document}

\begin{center}
  \LARGE
Fibered nonlinearities for $p(x)$-Laplace equations\\[5mm]
\large
Milena Chermisi and Enrico Valdinoci\\[5mm]

{\small MC: {\em    
NWF I -- Mathematik,
Universit\"at Regensburg,
D-93040 Regensburg       
Germany}
\\[5mm]
EV:
{\em
Dipartimento di Matematica,
Universit\'a di Roma Tor Vergata,
I-00133 Roma Italy}
\\[5mm]
{\tt milena.chermisi@mathematik.uni-regensburg.de}\\
{\tt enrico.valdinoci@uniroma2.it}
}\\[5mm]

\today\\[5mm]

{\small
}
\\[9mm]

\begin{minipage}{0.8\textwidth}
{\footnotesize {\bf Abstract.}
In $\R^m\times\R^{n-m}$, endowed with
coordinates~$X=(x,y)$, we consider
the PDE
$$ -{\rm div}\, \big(
\alpha(\x) |\nabla u(\X)|^{p(x)-2}\nabla u(\X)\big)=f(x,u(\X)).$$
We prove a geometric inequality and 
a symmetry result.}\\[5mm]
\end{minipage}

{\footnotesize EV is supported by 
MIUR, project
``Variational methods and Nonlinear Differential 
Equations''.}
\\[9mm]
\end{center}

\section{Introduction}

The purpose of this paper is to give some geometric results 
on the following problem:
\begin{equation}\label{eq1-provv}
-{\rm div}\, \big(
\alpha(\x) |\nabla u(\X)|^{p(x)-2}\nabla u(\X)\big)=f(x,u(\X))
\qquad 
{\mbox{ in $\Omega$,}} 
\end{equation}
where $f=f(x,u)\in L^\infty(\R^m \times\R)$ is differentiable
in $u$ with $f_{u}\in 
L^\infty(\R)$, $\alpha\in L^\infty(\R^m)$, with $
\displaystyle\inf_{\R^m}\alpha
>0$, $p\in L^\infty(\R^m)$, with $p(x)\ge 2$
for any $x\in\R^m$, and 
$\Omega$ is an open subset of $\R^n$.
 
Here,
$u=u(\X)$, with~$\X=(\x,\y)\in\R^m\times\R^{n-m}$.

As well known, the operator in \eqref{eq1-provv}
comprises, as main example, the degenerate $p(x)$-Laplacian
(and, in particular, the degenerate $p$-Laplacian).

The motivation of this paper is the following.
In \cite{DeG}, it was asked whether or not the level
sets of bounded,
monotone,
global solutions of
\begin{equation}\label{ORI}
-\Delta u(X)=
u(X)-u^3(X)\end{equation}
for $X\in\R^n$,
are flat hyperplanes,
at least when $n\le 8$.

In spite of the marvelous progress
performed in this direction (see, in particular,
\cite{Modica, BCNPisa, GG1, GG2, AC, AAC,
savin, Delpino}), part of the conjecture
and many related problems are
still unsolved (see \cite{FVre}).

In \cite{OVcil}, the following generalization
of \eqref{ORI} was taken into account:
\begin{equation}\label{ORI2}
-\Delta u(X)= f(x,u(X)).
\end{equation}
where, as above, the notation
$X=(x,y)\in\R^m \times\R^{n-m}$ is used.

We observe that
when~$f(x,u)$ does not depend on $x$, then~\eqref{ORI2}
reduces to a usual semilinar equation,
of which \eqref{ORI} represents the chief example.

When $f(x,u)$ depends on $x$,
the dependence on the space variable of $f$
changes
only with respect to a subset of the variables,
namely the nonlinearity takes no dependence on $y$.

In particular, for fixed $u\in\R$, we have that $f(x,u)$
is constant on the
``vertical fibers'' $\{x=c\}$, and for this
the nonlinearity in \eqref{ORI2} is called
``fibered''.

Moreover, the model in \eqref{ORI2} was considered in
\cite{OVcil} as a sort of interpolation between
the classical semilinear equation in \eqref{ORI}
and the boundary reactions PDEs of \cite{CSM, SV}, which
are related to fractional power operators (see also \cite{cafS}).

The purpose of this paper is to extend the results
of \cite{OVcil} to degenerate operators of $p(x)$-Laplace type
and thus replace \eqref{ORI2} with the more
general PDE in \eqref{eq1-provv}.
Indeed, when $p(x)$ is identically equal to
$2$, \eqref{eq1-provv}
was dealt with in \cite{OVcil}. Here, further technical
difficulties arises when $p(x)>2$, due to the presence
of a degenerate operator. To overcome these
difficulties, the technique developed in \cite{SV2}
will turn out to be useful.

We recall that the
$p(x)$-Laplace equations have recently
become quite popular, in view of some important physical
applications:
see, for instance, \cite{APP87, APP92,
APP00, APP02, APP05}. 

Moreover, many analytical results
related to the $p(x)$-Laplacian operator
have been recently appeared:
see, among the others, \cite{R99, R02a, R02b, R02c, R03a,
R03, R04a, R04b, R05, R05-1, R05-2, R06a, R06b, R06c,
R06d, R06e, R06, R07-1, R07-2, R07-3, R07-4, R08a, R08b}.

For us, a weak solution of \eqref{eq1-provv} is a function 
$u$ satisfying
\begin{equation}\label{eq1} 
\int_{{\Omega}} 
\alpha(\x) |\nabla u|^{p(x)-2}\nabla u\cdot 
\nabla\xi \, d\X
=
\int_{\Omega} 
f(x,u)\xi \, d\X 
\end{equation} 
for any $\xi\in C^\infty_0(\Omega)$.

In what follows, we always assume that
\begin{equation}\label{REGO}
{\mbox{$u\in C^1(\Omega)\cap C^2(\Omega\cap
\{ \nabla u\ne0\})\cap 
L^\infty(\Omega)$
and that $\nabla u\in L^\infty(\Omega) \cap W^{1,2}_{\rm loc}(\Omega)$.
}}\end{equation}
We recall that these regularity assumptions are very mild,
and automatically
fulfilled in many cases of interest
(see, for instance, \cite{DiBenede,Tol,DaSc}
and the discussion after Theorem 1.1
in \cite{FSV}).

In the sequel, we consider the
map $\mathcal{B}: \R^m \times \big(
\R^{n} \backslash \left \{ 0 \right \} \big)
\rightarrow {\rm{Mat}}
(n\times n)$ given by
\begin{equation}
\label{BDE}
\mathcal{B}(\x,\eta)_{ij}:=
\alpha(\x) |\eta|^{p(x)-2}
\left(
\delta_{ij}+(p(x)-2) \frac{\eta_i \eta_j}{|\eta|^2}
\right)
\end{equation}
for any $1 \leq i,j \leq n$, where $ {\rm{Mat}}
(n\times n)$ denotes the space of square $(n\times n)$-matrices.

We also extend this definition by continuity, setting
$\mathcal{B}(\x,0)_{ij}:=\alpha(x)\delta_{ij}$
when $p(x)=2$ and
$\mathcal{B}(\x,0)_{ij}:=0$ when $p(x) >2$.

We remark that
\begin{equation}\label{6bis}
\frac{d}{d\varepsilon}\Big[ 
\alpha(\x) |\nabla u +\varepsilon
\nabla \varphi|^{p(x)-2} (\nabla u +\varepsilon
\nabla \varphi)\cdot \nabla \varphi \Big]_{
\varepsilon=0}=<\mathcal{B}
(\x,\nabla u) \nabla \varphi,\nabla \varphi>
\end{equation}
for any smooth test function $\varphi$,
where $<,>$ denotes
the standard scalar product in $\R^{n}$. 

In view of~\eqref{6bis}, it is natural to say that~$u$ is 
stable
if
\begin{equation}\label{sta1} 
\int_{\Omega} <\mathcal{B}(\x,\nabla u)\nabla \xi, \nabla \xi> 
-
f_{u}(x,u)\xi^2\, d\X\,\ge\,0 
\end{equation} 
for any $\xi\in C^\infty_0(\Omega)$. 

The notion of
stability given 
in~\eqref{sta1} appears naturally in the calculus 
of variations setting and it 
is usually related to minimization 
and monotonicity properties. 
In particular, \eqref{6bis} and~\eqref{sta1} 
state that the (formal) second variation 
of the energy functional associated 
to the equation has a sign (see, e.g.,~\cite{Moss, FCS, AAC, FSV} and
Lemmata \ref{79pre}
and \ref{79} here for further details). 
 
The main results we prove are a geometric formula, 
of Poincar\'e-type, given in Theorem~\ref{POIN:TH}, 
and a symmetry result, given in Theorem~\ref{SYM:TH}. 
 
For our geometric result, we need to recall 
the following notation. Fixed $x\in\R^m$ and~$c\in\R$, we 
look at the level set 
$$ S:= \{  
\y\in\R^{n-m} \, : \, 
u(\x,\y)=c 
\}.$$ 
We will consider the regular points of~$S$, 
that is, we define 
$$ L:=\{ \y\in S \, : \, 
\nabla_\y u(\y,\x)\neq 0 
\}.$$ 
Note that~$L$ depends on the~$\x\in\R^m$ 
that we have
fixed at the beginning, though we do not keep 
explicit track of this in the notation. In the same way,
$S$ has to be thought as the level set of $u$ on
the slice selected by the fixed $x$.
 
Let $\nabla_L$ to be the tangential gradient 
along~$L$, that is, for any~$\y_o\in L$ 
and any~$G:\R^{n-m}\rightarrow\R$ smooth in the vicinity of~$\y_o$, 
we set 
\begin{equation}
\label{GR} \nabla_L G(\y_o):= 
\nabla_\y G(\y_o)-\left(\nabla_\y G(\y_o)\cdot 
\frac{\nabla_\y u(\x,\y_o)}{| 
\nabla_\y u(\x,\y_o)|}\right) 
\frac{\nabla_\y u(\x, \y_o)}{| 
\nabla_\y u(\x, \y_o)|}.
\end{equation}
Since~$L$ is a smooth $(n-m-1)$-manifold, in virtue of 
the Implicit Function Theorem and \eqref{REGO},
we can define 
the principal curvatures on it, denoted by 
$$\kappa_1(\x,\y),\dots, 
\kappa_{n-m-1}(\x,\y),$$ for any~$\y\in L$. 
We will then define the total curvature 
$$ {\mathcal{K}}(\x,\y):=\sqrt{ 
\sum_{j=1}^{n-m-1} \big(\kappa_j (\x,\y)\big)^2 
}.$$ 
 
Here is the geometric formula we prove in this paper:

\begin{theorem}\label{POIN:TH}
Let  $\Omega\subseteq\R^n$ be an open set.
Assume that 
$u$ is a stable weak 
solution of~\eqref{eq1-provv} in $\Omega$
under assumption \eqref{REGO}.

Then,
\begin{equation}\label{rtyua77a7a}
\begin{split}
&\int_{ {\mathcal{R}} } \alpha(x) |\nabla u|^{p(x)-2}
\left(\mathcal{S}+
\mathcal{K}^2|\nabla_y u|^2 +|\nabla_L |\nabla_y u||^2
+\frac{(p(x)-2)}{|\nabla u|^2} \mathcal{T}\right)\phi^2
\\ &\qquad\quad\le
\int_\Omega
|\nabla_y u
|^2
<\mathcal{B}(x,\nabla u)\nabla \phi,\nabla \phi > 
\end{split}
\end{equation}
for any $\phi\in
C^\infty_0$, where
\begin{equation}\label{WR}
{\mathcal{R}}:=\{
(\x,\y)\in\Omega\subseteq\R^m\times\R^{n-m} \, : \,
\nabla_\y u(\x,\y)\neq 0
\},\end{equation}
\begin{equation}\label{D.S.}
\mathcal{S}:=
-|\nabla_x |\nabla_y u||^2+
\sum_{i=1}^m \sum_{j=1}^{n-m}(u_{x_i y_j})^2
\qquad
\quad{\mbox{
and}}\end{equation}
\begin{equation}\label{D.T.}
\mathcal{T}:=-(\nabla u \cdot \nabla |\nabla_y u 
|)^2+\sum_{j=1}^{n-m}
(\nabla u \cdot \nabla u_{y_j})^2.\end{equation}
Also
\begin{equation}\label{last claim 1}
{\mbox{$\mathcal{S}$, $\mathcal{T}\ge0$ on ${\mathcal{R}}$}}
\end{equation}
and
\begin{equation}\label{last claim 2}
\begin{split}
& {\mbox{$\mathcal{S}(X)=0$ at some $X\in\R^n$}}
\\&{\mbox{
if and only if
$\nabla_y u_{x_i}(X)$ is parallel to $\nabla_y u(X)$}}
\\ &{\mbox{
for any $i=1,\dots,m$.
}}
\end{split}\end{equation}
\end{theorem} 

The second result we present is a symmetry result:
 
\begin{theorem}\label{SYM:TH}
Let 
$u$ be a weak
solution of~\eqref{eq1-provv} in whole $\R^n$
under assumption \eqref{REGO} (with $\Omega:=\R^n$ in \eqref{REGO}).

Suppose that
\begin{equation}\label{Monotone}
{\mbox{$\partial_{y_1} u(X)>0$ for any $X\in\R^n$,}}
\end{equation}
and that
there exists~$C_o\geq 1$ in such a way that
\begin{equation}\label{en:bound}
\int_{B_R}  
\alpha(x)|\nabla u|^{p(x)}\,dX
\le C_o\,
R^2,\end{equation}
for any~$R\ge C_o$.

Then, there exist~$\omega\in
{\rm S}^{n-m-1}$
and~$u_o: \R^m\times \R\rightarrow\R$
such that
$$ u(x,y)=u_o(x,\omega\cdot y)$$
for any~$(x,y)\in \R^m\times\R^{n-m}$.
\end{theorem} 

For explicit conditions that imply
the energy bound in \eqref{en:bound},
we refer to Appendix \ref{Motivation}
here below.

We observe that
Theorem~\ref{POIN:TH}
may be seen as a weighted
Poincar\'e 
inequality. Namely, the $L^2$-norm of
any test functions is bounded by the $L^2$-norm
of its gradient, but these norms are
taken with appropriate weights.

Remarkably, such weights have nice geometric meanings,
which make Theorem~\ref{POIN:TH}
feasible for the application in
Theorem~\ref{SYM:TH}, which
is related to the problem posed in \cite{DeG}
settled for the PDE in~\eqref{eq1-provv} instead of the one
in~\eqref{ORI}.

We recall that~\cite{SZ1, SZ2} introduced
a similar weighted Poincar\'e
inequality
in the classical uniformly elliptic
semilinear framework.
The idea of making use of Poincar\'e type
inequalities on level sets to deduce
suitable symmetries for the solutions was already in~\cite{FAR}
and it
has been also used in~\cite{CabCap, FSV}.

For related
{S}obolev-{P}oincar\'e inequalities, see \cite{Ferrari}.

We remark that results analogous to 
Theorems~\ref{POIN:TH}
and~\ref{SYM:TH}
hold, with the same proofs we present in this paper,
even for slightly more general degenerate
operators. For example,
the arguments we perform here also work
when~\eqref{eq1-provv}
is replaced by
$$
-{\rm div}\, \Big(
a \big(
\x, |\nabla u(\X)|
\big) \nabla u(\X)\Big)=f(x,u(\X))
,$$ with $0\le a\in L^{\infty}(\R^m\times [0,+\infty))$,
$\displaystyle\inf_{x\in \R^m} a(x,t)>0$
for any $t>0$ and $0\le a_t \in L^{\infty}(\R^m\times [0,+\infty))$.

The rest of the paper is devoted to the proofs 
of Theorems~\ref{POIN:TH} and~\ref{SYM:TH}, which
will be given in Sections~\ref{Section 1}
and~\ref{Section 2} respectively.
The paper ends with an Appendix, which contains some
auxiliary lemmata, some comments on when
conditions~\eqref{Monotone}
and~\eqref{en:bound} are satisfied, and explicit examples
of smooth, global, bounded solutions of~\eqref{eq1-provv}.

\section{Proof of Theorem \ref{POIN:TH}}\label{Section 1}
 
By~\eqref{BDE}, we have that 
\begin{equation}\label{26bis}
\begin{split}
&\int_{\Omega} \alpha(x)|
\nabla u|^{p(x)-2}\nabla u \cdot \Psi_{y_j}
\\&
=-\int_{\Omega } \left(
\alpha(x)|\nabla u|^{p(x)-2} \nabla u_{y_j} 
\cdot \Psi + (p(x)-2)\alpha(x)|\nabla u|^{p(x)-2} 
\frac{\nabla u \cdot \nabla u_{y_j}}{|\nabla u|^2} \nabla u \cdot \Psi
\right) \\
&
=-\int_{\Omega } <\mathcal{B}(x,\nabla u) \nabla u_{y_j}, \Psi >. 
\end{split}\end{equation}
for any~$j=1,\dots, n-m$ and any~$\Psi\in C^\infty_0 
(\Omega, \R^{n-m})$. 
 
The use of~\eqref{eq1} and~\eqref{26bis} with~$\Psi:=\nabla\psi$
yields
\begin{equation}\label{a711aa} 
\begin{split} 
&
\int_{\Omega} f_{u}(x,u)  u_{\y_j} \psi=
\int_{\Omega} (f(x,u))_{y_j}\psi
=-\int_{\Omega} f(x,u)\psi_{y_j}=
-\int_\Omega \alpha(x) |\nabla u|^{p(x)-2}
\nabla u \cdot \nabla \psi_{y_j} 
\\&\qquad\qquad=\int_{\Omega  } 
<\mathcal{B}(\x,\nabla u) \nabla u_{\y_j}, 
\nabla \psi >
\end{split} 
\end{equation} 
for any~$j=1,\dots, n-m$ and any~$\psi\in 
C^\infty_0(\Omega)$. 

Actually, 
\begin{equation}\label{dense}
{\mbox{
\eqref{a711aa} holds for any 
$\psi \in W^{1,2}_0(\Omega)$.}}\end{equation}
To prove \eqref{dense}, we perform a density argument
(which may be skipped by the expert reader).
Namely, we take $K$ to be a compact subset of $\Omega$,
$\psi \in 
W^{1,2}_0(K)$ and a sequence $\psi_\eps\in C^\infty_0(K)$
approaching $\psi$ in the $W^{1,2}$-norm.

We observe that, from \eqref{REGO},
there exists $C_K\ge1$ such that
\begin{equation}\label{OM}
\sup_{X\in K} |f_u(x,u(X))|+|\nabla_y u(X)|
+ |\mathcal{B} (x,\nabla u(X))| \le C_K.
\end{equation}
Furthermore, $\mathcal{B}$ is nonnegative definite.

Consequently, by
Cauchy-Schwarz
inequality,
\begin{equation}\label{DD1}
\begin{split}
& \left| \int_\Omega <\mathcal{B}(x,\nabla u) \nabla u_{y_j},
\nabla (\psi-\psi_\eps) >\right| 
\\ &\qquad\le
\sqrt{\int_K
<\mathcal{B}(x,\nabla u) \nabla u_{y_j},
\nabla u_{y_j} > 
}\,
\sqrt{\int_K
<\mathcal{B}(x,\nabla u) \nabla(\psi-\psi_\eps)
,\nabla (\psi-\psi_\eps)> }
\\ &\qquad\le
C_K^2 \, \sqrt{|K|}\,
\| \nabla(\psi-\psi_\eps)\|_{L^2(K)}
.
\end{split}
\end{equation}
Moreover,
\begin{equation}\label{DD2}
\left|\int_\Omega f_u (x,u) u_{y_j}(\psi-\psi_\eps)
\right|
\le C_K^2\,\int_K |\psi-\psi_\eps|
\le C_K^2\,\sqrt{|K|}\, \| \psi-\psi_\eps\|_{L^2(K)}.
\end{equation}
Then, \eqref{dense} plainly follows from \eqref{DD1}
and \eqref{DD2}.

We also claim that
\begin{equation}\label{dense2}
{\mbox{
\eqref{sta1} holds for any
$\xi \in W^{1,2}_0(\Omega)$.}}\end{equation}
The proof of \eqref{dense2} is analogous to the one
of \eqref{dense} and its reading may be omitted by
the expert
readers. The details of the proof of \eqref{dense2}
consist in taking a compact subset $K$ of $\Omega$,
a function $\xi\in W^{1,2}_0(K)$,
and a sequence $\xi_\eps\in C^\infty_0(K)$ which approaches
$\xi$ in the $W^{1,2}$-norm.

Then, using \eqref{OM} once more,
\begin{eqnarray*}
&& \left| \int_{\Omega} \Big( <\mathcal{B}(\x,\nabla u)\nabla \xi, \nabla 
\xi>-
<\mathcal{B}(\x,\nabla u)\nabla \xi_\eps, \nabla \xi_\eps>\Big)\right|
\\ &&\qquad +\left|\int_\Omega
\big( f_{u}(x,u)(\xi^2 -\xi_\eps^2)\big)
\right|
\\
&\le& \left| \int_{K} \Big( <\mathcal{B}(\x,\nabla u)\nabla (
\xi-\xi_\eps), 
\nabla
\xi>+
<\mathcal{B}(\x,\nabla u)\nabla \xi_\eps, \nabla 
(\xi-\xi_\eps)>\Big)\right|
\\ &&\qquad +C_K\, \int_K
\big( |\xi+\xi_\eps| |\xi-\xi_\eps| \big)
\\ &\le& 4 C_K 
\,\big( 1+
\|\xi\|_{W^{1,2}(K)}
\big)
\|\xi-\xi_\eps\|_{W^{1,2}(K)}
,\end{eqnarray*}
for $\eps$ small, and this proves \eqref{dense2}. 

{F}rom \eqref{REGO} and
\eqref{dense}, we may take~$\psi:=u_{y_j} 
\phi^2$ in \eqref{a711aa},
where~$\phi\in C^\infty_0(\Omega)$. 

So, we obtain
\begin{equation}\label{3.2bis}
\begin{split}
& 0= \int_{\Omega}
\left[<\mathcal{B}(x,\nabla u) \nabla u_{y_j},
\nabla u_{y_j} > \phi^2+ 
<\mathcal{B}(x,\nabla u) \nabla u_{y_j}, \nabla \phi^2> u_{y_j}\right]\\
& =\int_{\Omega} f_{u}(x,u) u_{y_j}^2 \phi^2.
\end{split}
\end{equation}

Now, we notice that, by \eqref{REGO} and Stampacchia's
Theorem (see, e.g., Theorem~6.19
in~\cite{LOSS}), 
\begin{equation}\label{ASTA}\begin{split}
& \nabla |\nabla_y u| =0=\nabla u_{y_j}
\\
&{\mbox{for a.e. $\x\in\R^m$ and a.e. $y\in \R^{n-m}$
such that~$\nabla_y u(x,y)=0$.}}
\end{split}\end{equation}

By~\eqref{WR},
\eqref{3.2bis} and~\eqref{ASTA}, we obtain
\begin{equation*}
0= \int_{\RR} 
\left[<\mathcal{B}(x,\nabla u) \nabla u_{y_j},
\nabla u_{y_j} > \phi^2+ 
<\mathcal{B}(x,\nabla u) \nabla u_{y_j}, \nabla \phi^2> u_{y_j}\right]+
\int_{\Omega} f_{u}(x,u) u_{y_j}^2 \phi^2. 
\end{equation*} 

We now sum over $j=1,...,n$ to get (dropping, for short,
the dependences of $\mathcal{B}$) and
we obtain
\begin{equation} \label{78iddudududuudaa}
-\int_{\RR} \left[\sum_{j=1}^n <\mathcal{B}\nabla u_{y_j}, \nabla u_{y_j} >
\phi^2- \frac{1}{2}<\mathcal{B} \nabla |\nabla_y u|^2, \nabla \phi^2>
\right] =\int_{\Omega} f_{u}(x,u) |\nabla_y u|^2 \phi^2. 
\end{equation} 

Now, we recall \eqref{dense2}
and we choose~$\xi:=|\nabla_y u|\phi$
in~\eqref{sta1}, obtaining
\begin{equation*}\begin{split}
& 0 \leq \int_{\RR}\Big[ <\mathcal{B}
\nabla |\nabla_y u|, \nabla |\nabla_y u|> \phi^2 
+ <\mathcal{B} \nabla \phi, \nabla \phi> |\nabla_y u|^2\Big.\\
& \Big. +2 <\mathcal{B}\nabla |\nabla_y u| ,
\nabla \phi> |\nabla_y u| \phi \Big]
+ \int_{\Omega} f_{u}(x,u) |\nabla_y u|\phi^2,
\end{split} 
\end{equation*}
where~\eqref{ASTA} has been used once more.

This and~\eqref{78iddudududuudaa}
imply that 
\begin{equation}
\label{s88818181} 
0 \leq \int_{\RR}\Big[ <\mathcal{B} \nabla |\nabla_y u|, 
\nabla |\nabla_y u|> \phi^2 +
<\mathcal{B} \nabla \phi, \nabla \phi> |\nabla_y u|^2 
-\sum_{j=1}^n <\mathcal{B}\nabla u_{y_j}, \nabla u_{y_j} > \phi^2
\Big]. 
\end{equation} 

By using \eqref{BDE} and~\eqref{s88818181}, we are lead
to the following inequality:
\begin{equation}\label{31jkl}
\begin{split}
& 0 \leq \int_{\mathcal{R}} \Big\{
\alpha(x)|\nabla u|^{p(x)-2} \phi^2 
\Big [|\nabla |\nabla_y u||^2
-\sum_{j=1}^{n-m}  |\nabla u_{y_j}|^2\Big ]+
< \mathcal{B} \nabla \phi, \nabla \phi> |\nabla_y 
u|^2 \Big.\\
&\quad\quad \Big.
+\frac{(p(x)-2)\alpha(x)|\nabla u|^{p(x)-2} \phi^2}{|\nabla u|^2} 
\Big [(\nabla u \cdot \nabla |\nabla_y 
u|)^2 -\sum_{j=1}^{n-m} (\nabla u \cdot \nabla u_{y_j})^2 \Big ]\Big\}
.\end{split}
\end{equation}

We denote $\mathcal{S}$ and $\mathcal{T}$ as in \eqref{D.S.}
and \eqref{D.T.}. 

We also set
$$ \mathcal{U}:= \big| \nabla|\nabla_y u| \big|^2-\sum_{j=1}^{
n-m}|\nabla u_{y_j}|^2.$$
Making use of formula~(2.1)
of~\cite{SZ1}, we have that, on~${\mathcal{R}}$,
\begin{equation*}
\mathcal{U}+\mathcal{S}=
\big| \nabla_y |\nabla_y u| \big|^2-\sum_{i,j=1}^{
n-m} (u_{y_i y_j})^2
=
-(\mathcal{K}^2
|\nabla_y u|^2+|\nabla_L |\nabla_y u||^2).
\end{equation*}
Accordingly,~\eqref{31jkl}
becomes
\begin{eqnarray}
\label{yuiooaoo}
\nonumber
0 \leq \int_{ {\mathcal{R}} } \Big\{
\alpha(x)|\nabla u|^{p(x)-2} \phi^2 \Big (-\mathcal{S}
-(\mathcal{K}^2|\nabla_y u|^2 +|\nabla_L |\nabla_y u||^2)\Big )\Big.\\
\Big.
-\frac{(p(x)-2)\alpha(x)|\nabla u|^{p(x)-2}}{
|\nabla u|^2} \mathcal{T}\phi^2
+<\mathcal{B}\nabla \phi,\nabla \phi > |\nabla_y u |^2\Big\},\nonumber
\end{eqnarray}
and this gives \eqref{rtyua77a7a}.

Furthermore,
if we set
$$ \zeta_j :=\nabla u\cdot \nabla u_{y_j}
\qquad{\mbox{ for $j=1,\dots, n-m$,}}$$
and
$$ \zeta:= (\zeta_1,\dots,\zeta_{n-m})\in \R^{n-m},$$
we have that, on ${\mathcal{R}}$,
\begin{equation}\label{-LC1}
\begin{split}
&  -{\mathcal{T}} =
\left( \sum_{\ell=1}^n \partial_\ell u
\partial_\ell
|\nabla_y u|\right)^2-|\xi|^2 =
\left( \sum_{\ell=1}^n \partial_\ell 
u
\frac{\nabla_y u}{|\nabla_y u|}\cdot \nabla_y \partial_\ell u
\right)^2-|\xi|^2\\
&\qquad =\left(
\frac{\nabla_y u}{|\nabla_y
u|}\cdot \xi
\right)^2-|\xi|^2
\le 0,\end{split}
\end{equation}
thanks to
Cauchy-Schwarz
inequality.

Analogously, for any $i=1,\dots,m$, on ${\mathcal{R}}$,
\begin{equation}\label{-LC0}
\big|\partial_{x_i}|\nabla_y u| \big|
=
\left|\frac{
\nabla_y u}{|\nabla_y u|}\cdot \nabla_y u_{x_i}
\right|\le
|\nabla_y u_{x_i}|
=\sqrt{\sum_{j=1}^{n-m} (u_{x_i y_j})^2}
,
\end{equation}
and
\begin{equation}\label{-LC2}
{\mbox{equality holds in \eqref{-LC0} if and
only if $\nabla_y u_{x_i}$ is parallel
to $\nabla_y u$. }}
\end{equation}
Therefore, from \eqref{-LC0},
\begin{eqnarray*}
&& -\mathcal{S} =
|\nabla_x |\nabla_y u||^2-
\sum_{i=1}^m \sum_{j=1}^{n-m}(u_{x_i y_j})^2
\\ &&\qquad\quad
=\sum_{i=1}^m \big(\partial_{x_i} |\nabla_y u|\big)^2
-
\sum_{i=1}^m \sum_{j=1}^{n-m}(u_{x_i y_j})^2
\le 0.
\end{eqnarray*}
This, \eqref{-LC1} and \eqref{-LC2} give
\eqref{last claim 1} and \eqref{last claim 2},
thus completing the proof of 
Theorem~\ref{POIN:TH}.~\CVD 
 
\section{Proof of Theorem \ref{SYM:TH}}\label{Section 2}
 
{F}rom \eqref{Monotone}
and Lemma \ref{79}, we have that $u$ is stable.
Therefore, the assumptions of Theorem \ref{POIN:TH}
are implied by the ones of
Theorem~\ref{SYM:TH}.

Given~$\rho_1\le\rho_2$, we define 
\begin{equation}\label{-A}
{\mathcal{A}}_{\rho_1,\rho_2}:=\{ 
X\in\R^n \, : \,
|X|\in [\rho_1,\rho_2] 
\}.\end{equation}
 
{F}rom \eqref{en:bound} and Lemma~\ref{tatay}, applied here with
$$h(X):=\alpha(x)|\nabla u|^{p(x)},$$
we obtain
\begin{equation}\label{7s77s88} 
\int_{{\mathcal{A}}_{\sqrt R, R}}\frac{\alpha(x)
|\nabla u|^{p(x)} 
}{ |X|^2} \leq C_1\log R 
\end{equation} 
for a suitable~$C_1>0$, if~$R$ is big. 
 
Now we define 
$$ \phi_R(X):=\left\{ 
\begin{matrix} 
\log R & {\mbox{ if $|X|\le \sqrt R$,}}\\ 
2\log\big( R/|X|\big) 
& {\mbox{ if $\sqrt R<|X|< R$,}} 
\\ 
0 & {\mbox{ if $|X|\ge R$}} 
\end{matrix} 
\right.$$ 
and we observe that 
$$
|\nabla\phi_R|\leq \frac{C_2\,\chi_{ 
{\mathcal{A}}_{\sqrt R, R} 
}}{|X|},$$
for a suitable~$C_2>0$. 

Moreover,
employing \eqref{BDE} and Cauchy-Schwarz
inequality, 
$$ \big| < {\mathcal{B}}(x,\nabla u(x))w, w>\big|
\,\le
\,
\alpha(x)\,(p(x)-1)\,|\nabla u(x)|^{p(x)-2}|w|^2
\qquad
{\mbox{
for all $w\in\R^n$.}}$$

Thus, plugging~$\phi_R$ in \eqref{rtyua77a7a}
and recalling \eqref{last claim 1},
we see that
\begin{eqnarray*} 
&& (\log R)^2\int_{B_{\sqrt{R}}\bigcap 
{\mathcal{R}} } \left[ 
\alpha(x) |\nabla u|^{p(x)-2}
\left(\mathcal{S}+
\mathcal{K}^2|\nabla_y u|^2 +|\nabla_L |\nabla_y u||^2
\right)\right]
\\&&\qquad\qquad\,\leq\,C_3 
\int_{ 
{\mathcal{A}}_{\sqrt R, R} 
}\frac{ 
\alpha(x)|\nabla u|^{p(x)-2}
|\nabla_y u|^2}{|X|^2} 
\end{eqnarray*} 
for large~$R$.

Hence, we divide by~$(\log R)^2$, 
we use~\eqref{7s77s88} 
and we send~$R\rightarrow+\infty$. In this way,
we obtain
that ${\mathcal{S}}$,
${\mathcal{K}}$ and $\big| 
\nabla_L |\nabla_y u| 
\big|$ vanish identically
on~${\mathcal{R}}$. 
 
Then, 
by Lemma~2.11 of~\cite{FSV} (applied to the function~$y\mapsto 
u(x,y)$, for any fixed~$x\in \R^m$),
we obtain that
there exist $\omega:\R^m\rightarrow{\rm S}^{n-m-1}$
and $u_o:\R^m\times\R\rightarrow\R$ such that
$u(x,y)=
u_o (x,
\omega(x)\cdot y)$
for any $(x,y)\in\R^m\times\R^{n-m}$.

{F}rom \eqref{last claim 2}
and Lemma \ref{om-co}, we deduce that $\omega$
is constant, and this ends the proof of
Theorem~\ref{SYM:TH}.~\CVD 

\begin{appendices}

\section{Auxiliary lemmata}

\begin{lemma}\label{om-co}
Let $u\in C^2 (\R^n)$, with
\begin{equation}\label{6w812312}
\big\{ (x,y)\in \R^m\times\R^{n-m}
\, : \,
\nabla_y u(x,y)=0\big\}
=\emptyset.\end{equation}

Let also $\omega:\R^m\rightarrow{\rm S}^{n-m-1}$
and $u_o:\R^m\times\R\rightarrow\R$.

Suppose that
\begin{equation}\label{67}
u(x,y)=
u_o (x,
\omega(x)\cdot y)
\end{equation}
for any $(x,y)\in\R^m\times\R^{n-m}$.

Assume also that
\begin{equation}\label{86}
{\mbox{$\nabla_y u_i$ is parallel to $\nabla_y u$}}
\end{equation}
for any $i=1,\dots,m$ and any $(x,y)\in\R^n\times\R^{n-m}$.

Then, $\omega$ is constant.
\end{lemma}

\begin{proof}
To start, we claim that
\begin{equation}\label{68}
{\mbox{$\nabla_{y} u(x,y)$
is parallel to $\omega(x)$.}}
\end{equation}
To check this, we
let $\eta(x)\in
{\rm S}^{n-m-1}$ be orthogonal to $\omega(x)$
and we use \eqref{67} to get that
$$ u(x,y+t\eta(x))=u_o(x,\omega(x)\cdot y).$$
Therefore, by differentiating
with respect to $t$, 
$$ \nabla_{y} u(x,y)\cdot \eta(x)=0.$$
This proves \eqref{68}.

{F}rom \eqref{68}, we now write
\begin{equation}\label{89}
\nabla_{y} u(x,y)=c(x,y) \omega(x),\end{equation}
for some $c(x,y)\in \R$.

In fact, from \eqref{6w812312} and \eqref{89},
\begin{equation}\label{890}
{\mbox{$c(x,y)\ne0$ for all $(x,y)\in\R^n$.}}
\end{equation}
Also, from~\eqref{89},
\begin{equation}\label{GA}
{\mbox{the map $(x,y)\mapsto c(x,y)\omega(x)$
belongs to $C^1(\R^n)$.}}
\end{equation}
Hence,
\begin{equation}\label{TTT}
\Big( c(x,y) \omega(x)\Big)_i =\nabla_{y} u_i(x,y),
\end{equation}
for any $1\le i\le n$.

Since
$$ c^2 (x,y) = \big(c(x,y)\omega(x)\big)\cdot
\big(c(x,y)\omega(x)\big),$$
we deduce from~\eqref{GA}
that $c^2\in C^1(\R^n)$.

Thus, from~\eqref{890}, 
\begin{equation}\label{CDI}
c\in C^1(\R^n).\end{equation}

This, \eqref{89} and \eqref{890} imply that
\begin{equation}
\label{ODI}
\omega\in C^1(\R^m).
\end{equation}
So,
\begin{equation}\label{NO}
0= \left(\frac12\right)_i =\left( \frac{\omega(x)
\cdot\omega(x)}{2}\right)_i=
\omega_i(x)\cdot\omega(x),
\end{equation}
for any $1\le i\le m$.

Furthermore, by \eqref{89},
\eqref{86} and \eqref{68}, we have that
\begin{equation}\label{6a77111}
\Big( c(x,y) \omega(x)\Big)_i = 
\Big( \nabla_y u(x,y)\Big)_i=\nabla u_i(x,y)=
k^{(i)}(x,y) \omega(x),
\end{equation}
for some $k^{(i)}(x,y)\in\R$.

Then, making use of~\eqref{NO} twice,
we deduce from \eqref{6a77111} that
$$ 0= k^{(i)}(x,y) \omega(x)\cdot
\omega_i(x)=
\Big( c(x,y) \omega(x)\Big)_i \cdot \omega_i(x)
=c(x,y) \omega_i(x)
\cdot \omega_i(x)=
c(x,y) |\omega_i(x)|^2
,
$$
for any $1\le i\le m$.

Consequently, from \eqref{890}, we conclude
that $\omega_i (x)=0$ for any $1\le i\le m$.
\end{proof}

We remark that the result in Lemma \ref{om-co}
is, in general, false without condition \eqref{6w812312}.
To see this, let us consider the following example.
Let $m=1$, $n=3$, $\tau\in C^\infty(\R)$, with $\tau(x)=0$
for any $x\in [-1,1]$ and $\tau(x)>0$ for any $x\in\R\setminus[-1,1]$.

Let also $\omega\in C^\infty(\R,{\rm S}^1)$ be such that $\omega(x)=
(1,0)$ for any $x\le-1/2$ and $\omega(x)=(0,1)$ for any $x\ge1/2$.

Let $\gamma\in C^\infty(\R)$, and set
$$ u_o(x,r):= \tau(x) \gamma(r), \qquad
{\mbox{
for any $(x,r)\in\R\times\R$, and}}$$
$$ u(x,y):=\tau(x) \gamma(\omega(x)\cdot y)
,\qquad
{\mbox{
for any $(x,y)\in\R\times\R^2$.}}$$
Then, \eqref{67} holds true.

Moreover,
\begin{equation}\label{its}
\nabla_y u(x,y)=\gamma'(\omega(x)\cdot y) \tau(x)\,\omega(x).
\end{equation}
We also observe that
\begin{eqnarray*}
\partial_x \big( \tau(x)\omega(x)\big)
&=& \left\{
\begin{matrix}
(0,0) & {\mbox{ if $x\in (-1,1)$,}}\\
\tau'(x) \,(1,0) & {\mbox{ if $x\le-1$,}}\\
\tau'(x) \,(0,1)
& {\mbox{ if $x\ge 1$}}\\
\end{matrix}
\right. \\&=&
\tau'(x)\,\omega(x).
\end{eqnarray*}
As a consequence,
\begin{eqnarray*}\nabla_y u_1 (x,y)&=&
\gamma'(\omega(x)\cdot y) \tau'(x)\,\omega(x)+
\gamma''(\omega(x)\cdot y) (\omega'(x)\cdot
y)\tau(x)\,\omega(x)
\\ &=& \Big(
\gamma'(\omega(x)\cdot y) \tau'(x)+
\gamma''(\omega(x)\cdot y) (\omega'(x)\cdot
y)\tau(x)
\Big)\,\omega(x).\end{eqnarray*}
That is, $\nabla_y u_1$ is parallel to $\omega$
and so, by \eqref{its}, we have that \eqref{86}
holds true.

But \eqref{6w812312} and the claim of Lemma \ref{om-co} are not
satisfied.

\begin{lemma}\label{tatay}
Let the notation in \eqref{-A} hold.
 
Let~$R>0$ 
and~$h:B_R\subset\R^n \rightarrow\R$ be a nonnegative 
measurable function.  
 
For any~$\rho\in (0,R)$,
let 
$$ \eta(\rho):=2\int_{B_{\rho}} h(X)\,dX.$$ 
Then, 
$$\int_{{\mathcal{A}}_{\sqrt R, R}}\frac{h(X)}{|X|^2}\,dX 
\leq \int_{\sqrt R}^R t^{-3}\eta(t)\,dt+\frac{\eta(R)}{R^2}. 
$$ 
\end{lemma} 

\begin{proof} 
The argument we give here is a modification of the ones
on page~24 of~\cite{Simon} and page~403 of~\cite{GT}.

By Fubini's Theorem,
\begin{equation}
\begin{split}
& \int_{{\mathcal{A}}_{\sqrt R, R}}\frac{h(X)}{|X|^2}\,dX
=
\int_{{\mathcal{A}}_{\sqrt R, R}}
h(X)\left( \int_{|X|}^R 2t^{-3}\,dt +R^{-2}\right)
\,dX\\
&\qquad= 2 \int_{\sqrt{R}}^R \int_{{\mathcal{A}}_{\sqrt R, t}}
t^{-3} h(X)\,dX\,dt+
R^{-2}\int_{{\mathcal{A}}_{\sqrt R, R}}
h(X)\,dX
\\ &\qquad\le
\int_{\sqrt{R}}^R t^{-3}\eta(t)\,dt
+R^{-2}\eta(R).
\qedhere\end{split}\end{equation}
\end{proof}

\section{Motivating assumptions \eqref{sta1} and \eqref{en:bound}}\label{Motivation}

For $t_0\in \R$ fixed, we set
\begin{equation}\label{dF}
F(x,t):=\int_{t_0}^t f(x,s) \, ds.
\end{equation}
Given an open set $\Omega\subseteq \R^n$ , we define
\begin{equation*}\label{en_p}
{\mathcal{E}}_\Omega(v):=\int_\Omega 
\frac{\alpha(x) |\nabla u (\X)|^{p(x)}}{p(x)} - F(x,u(\X)) \, d\X.
\end{equation*}
It is well known that u is a local minimizer if for any bounded open 
set $U\subset \Omega$ we have ${\mathcal{E}}_U(u)$ is well-defined and finite,
and $${\mathcal{E}}_U(u)(u+\phi)\geq {\mathcal{E}}_U(u)$$
for any $\phi\in C^\infty_0(U)$.

\begin{lemma}\label{79pre}
Let $u$ be a local minimizer in some domain $\Omega$.
Then $u$ satisfies 
\eqref{eq1} and \eqref{sta1}.
\end{lemma}

\begin{proof}
We compute the first and second variation of
${\mathcal{E}}_\Omega$ with 
$U$
a bounded open subset of $\Omega$. We have
\begin{eqnarray*}
0&=&\left.\frac{d}{d\eps}{\mathcal{E}}_U(u+\eps\phi)\right|_{\epsilon=0}\\
&=&\int_\Omega \alpha(x)|\nabla u|^{p(x)-2}\nabla u \cdot \nabla \phi
-f(x,u) \phi \, d\X
\end{eqnarray*}
and 
\begin{eqnarray*}
0&\leq&\frac{d^2}{d\epsilon^2}
\left. {\mathcal{E}}_U (u+\epsilon \phi)\right|_{\epsilon=0}\\
&=& \int_{\Omega} 
<\B(x,\nabla u)\nabla\phi,\nabla\phi>
- f_{u}(x,u)\phi^2\, d\X,
\end{eqnarray*}
due to \eqref{6bis}.
\end{proof}

We now recall that monotonicity in one direction
implies stability:

\begin{lemma}\label{79}
Let $u$ be a weak solution of \eqref{eq1-provv} in~$\Omega$
and suppose that $\partial_{y_1} u>0$ in~$\Omega$.

Then, $u$ is stable, that is \eqref{sta1} holds.
\end{lemma}

\begin{proof}
Fix $\xi\in C^\infty_0(\Omega)$. 
In view of \eqref{dense}, we may use \eqref{a711aa} for $j=1$ and   
$\psi:=\displaystyle\frac{\xi^2}{u_{y_1}}\in W^{1,2}_0(\Omega)$. 

This yields that
\begin{eqnarray*}
&&  \int_\Omega f_u(x,u) \xi^2 \,dX\\
&=&
\int_\Omega f_u(x,u) u_{y_1} \psi \,dX \\&=&
\int_\Omega \Big[
\frac{2 \xi}{u_{y_1}} < \B(x, \nabla u) \nabla u_{y_1}, \nabla \xi>
-\frac{\xi^2}{(u_{y_1})^2}  <\B(x, \nabla u) \nabla u_{y_1}, \nabla 
u_{y_1}>
\Big]\,dX \\
&\leq& \int_\Omega <\B(x, \nabla u) \nabla \xi, \nabla \xi>
\,dX,\end{eqnarray*}
where in the last equation we used 
that \begin{equation*}
2 <\B(x, \nabla u) v,w >
\;\; \leq \;\; <\B(x, \nabla u) v,v >+<\B(x, \nabla u) w,w >, \qquad
\forall v,w\in \R^n.
\qedhere\end{equation*}
\end{proof}

We now give a sufficient condition for~\eqref{en:bound}
to hold:

\begin{lemma}\label{12}
Let $t_o:=-1$ in~\eqref{dF}.

Assume that~$F(x,t)\le0$
for any $x\in\R^m$ and any~$t\in\R$, $F(x,-1)=F(x,+1)=0$, and 
\begin{equation}\label{F:bound}
\sup_{
{x \in \R^m}\atop{|t|\leq 1}}
\big| F(x,t)\big|<+\infty.
\end{equation}

Let~$u\in W^{1,\infty} (\R^n, [-1,1])$ be a local minimum
in the whole~$\R^n$.

Then, there exists~$C>0$ such that
\begin{equation}\label{787}
\int_{B_R}\alpha(x) |\nabla u(X)|^{p(x)}\, d\X \le CR^{n-1},
\end{equation}
for any~$R>1$.

In particular, if also~$n\le 3$, then~\eqref{en:bound}
holds.
\end{lemma}

\begin{proof} 
We take~$R>1$, $h\in C^\infty(B_R)$,
with $h=-1$ in $B_{R-1}$, $h=1$ on~$\partial B_R$
and~$|\nabla h|\le 4$, and we set~$v(x):=\min\{u(x),h(x)\}$.

Then, since~$u$ is minimal, we have that
\begin{eqnarray*}
&& \inf_{x\in\R^m}\frac{1}{p(x)} 
\int_{B_R} \alpha(x) |\nabla u|^{p(x)}\,d\X
\le
{\mathcal{E}}_{B_R} (u)\le {\mathcal{E}}_{B_R}(v)
\\ &&\qquad\qquad=\int_{B_R\setminus B_{R-1}} 
\left(\frac{1}{p(x)} \alpha(x) |\nabla v|^{p(x)}
-F(x,v)\right)\, d\X
\\
&&\qquad\qquad\le
\int_{B_R\setminus B_{R-1}} \left[
\sup_{x\in\R^m}
\frac{1}{p(x)} \sup_{x\in\R^m}\alpha(x) \Big(|\nabla u|^{p(x)}
+|\nabla h|^{p(x)}\Big) + \sup_{\R^m\times [-1,1]} |F|
\right]\,d\X,
\end{eqnarray*}
which implies~\eqref{787}. 
\end{proof}

We would like to remark that
the nonlinearities of the type in \eqref{ORI} 
satisfy the assumptions
of Lemma~\ref{12}. The following is another criterion
for obtaining \eqref{en:bound}:

\begin{lemma}\label{CRY}
Suppose that $p(x)=Ãp$ is constant.

Let $u$ be a bounded
weak solution of \eqref{eq1-provv} in the whole $\R^n$.

Let
$$ I:=\big[ -\|u\|_{L^\infty(\R^m)},
\|u\|_{L^\infty(\R^m)}
\big].$$
Suppose that there exist $C_0>0$ and~$\sigma
\in [1,2]$ such that
\begin{equation}\label{chig}
\int_{B_R\subset \R^m} \left[\sup_{r\in I}
|f(x,r)|\right]
\,dx \le C_0 R^{m-\sigma},
\end{equation}
for any $R\ge C_0$.

Then, there exists $C_1>0$ for which
\begin{equation}\label{DRY}
\int_{B_R\subset\R^n}
\alpha(x) |\nabla u(X)|^{p}\,dX \le C_1 R^{n-\sigma},
\end{equation}
for any $R\ge C_1$.

In particular,
\eqref{en:bound} holds
\begin{itemize}
\item[{(P1)}] either if $n\le 3$
and $f(x,r)
=0$ for any $(x,r)=(x_1,\dots,x_m,r)\in\R^m
\times\R$ such that $|x_1|\ge C_2$,
\item[{(P2)}] or if~$m\ge 2$, $n\le4$
and
$f(x,r)
=0$ for any $(x,r)=(x_1,\dots,x_m,r)\in\R^m
\times\R$ such that $|x_1|+|x_2|\ge C_2$,
\end{itemize}
for some~$C_2>0$.
\end{lemma}

\begin{proof}
The last claim
plainly follows from \eqref{DRY} (taking~$\sigma:=1$
in case~(P1) holds and~$\sigma:=2$
in case~(P2) holds).

Let us now prove
\eqref{DRY}.

For this, we define
$$ M:= 1+
\|u\|_{L^\infty (\R^n)}+
\|a\|_{L^\infty (\R^m)}+
\sup_{ {x\in\R^m}\atop{|r|\le \|u\|_{L^\infty (\R^n)} }}
|f(x,r)|.$$
We take $R\ge \max\{C_0, 1\}$
and we choose $\tau\in C^\infty_0 (B_{2R}, [0,1])$,
with $\tau=1$ in $B_R$ and $|\nabla\tau|\le 4/R$.

We also observe that, by
a scaled Young inequality,
\begin{equation}\label{scaled}
\begin{split}
& M p \alpha(x) \tau^{p-1} |\nabla u|^{p-1}\,
|\nabla \tau|
=
\big((\alpha(x))^{(p-1)/p}\tau^{p-1}
|\nabla u|^{p-1}
\big)\, \big(Mp (\alpha(x))^{1/p}
|\nabla \tau|\big)
\\ &\qquad\le
\frac12
\big((\alpha(x))^{(p-1)/p}\tau^{p-1}
|\nabla u|^{p-1}
\big)^{p/(p-1)}+C_3
\big((\alpha(x))^{1/p}
|\nabla \tau|\big)^p\\ &\qquad
=
\frac12
\alpha(x) \tau^{p}
|\nabla u|^{p}
+C_3 \alpha(x)
|\nabla \tau|^p,
\end{split}
\end{equation}
for a suitable $C_3>0$.

Then, using \eqref{eq1} and \eqref{scaled},
\begin{eqnarray*}
&& \int_{B_{2R}} \alpha(x) \tau^{p} |\nabla u|^{p}\,dX \\
&=&
\int_{B_{2R}} \alpha(x)
|\nabla u|^{p-2}\nabla u\cdot\nabla(\tau^{p} u)
- p\alpha(x) u \tau^{p-1}|\nabla u|^{p-2}
\nabla u
\cdot\nabla \tau\,dX \\&\le&
\int_{B_{2R}} \big| f(x,u) \tau^{p} u
\big|+M p \alpha(x) \tau^{p-1} |\nabla u|^{p-1}\,
|\nabla \tau|\,dX\\
&\le& M \int_{B_{2R}} \left[
\sup_{r\in I}
|f(x,r)|\right]
\,dX\\&&\qquad+\frac{1}{2}\int_{B_{2R}}\alpha(x)
\tau^{p} |\nabla u|^{p}\,dX
\\ &&\qquad
+C_3\int_{B_{2R}} \alpha(x)
|\nabla\tau|^{p}
\,dX.\end{eqnarray*}

This and \eqref{chig}
give that
\begin{eqnarray*}
\frac12 
\int\limits_{B_R\subset \R^n }
\alpha(x) |\nabla u|^{p}\,dX
&\le&
\frac12
\int\limits_{B_{2R}
\subset \R^n
} \alpha(x) \tau^{p} |\nabla u|^{p}\,dX
\\
&\le&
M \int\limits_{B_{2R}
\subset \R^{n-m}
} \left\{ \int\limits_{B_{2R}
\subset \R^m
}\left[
\sup_{r\in I} |f(x,r)|\right]\,dx\right\} \,dy
\\ &&\qquad
+C_3 \int\limits_{B_{2R}
\subset \R^n
} 
\alpha(x) |\nabla\tau|^{p}\,dX
\\ &\le&
C_0 M \int\limits_{B_{2R}
\subset \R^{n-m}
} R^{m-\sigma}
\,dy
+C_4 \int\limits_{B_R
\subset \R^n
}\frac{1}{R^{p}}\,dX\\
&=&
C_5 R^{m-\sigma} R^{n-m}
+C_6 R^{n-p},
\end{eqnarray*}
for suitable $C_4$, $C_5$, $C_6>0$.

This completes the proof of \eqref{DRY}.
\end{proof}

\section{An explicit example}

We would like to point out that
it is very easy to construct global, bounded, smooth
solutions of
\eqref{eq1-provv}.

For this, we take $\beta\in C^\infty(\R^m)\cap L^\infty(\R^m)$,
with
\begin{equation}\label{B+}
\inf_{\R^m}\beta>0.\end{equation}
Let also $\gamma\in 
C^\infty(\R)\cap L^\infty(\R)$. Assume that $\gamma$
is strictly increasing and let $\Gamma$ its inverse,
that is
\begin{equation}\label{G-} \Gamma\big( \gamma(t)\big)=t\qquad
{\mbox{ for any $t\in\R$.}}
\end{equation}

We fix $\omega\in {\rm S}^{n-m-1}$, and define
$$ u(x,y):= \beta(x)\gamma(\omega\cdot y).$$
We also define $g:\R^m\times \R$ to be
$$ g(x, \omega\cdot y):=
-{\rm div}\, \big(
\alpha(\x) |\nabla u(\X)|^{p(x)-2}\nabla u(\X)\big).$$
Also, for any $x\in\R^m$ and any $r\in\R$,
we set
$$ f(x,r):= g\Big( x, \Gamma \big(r/\beta(x)\big)\Big).$$
Notice that this definition is well posed, due to \eqref{B+}.

Then, recalling \eqref{G-}, it is easy to check
that $u$ is a solution of
\eqref{eq1-provv}.

\end{appendices}
\bibliographystyle{amsplain}
\bibliography{bibliofile}

\end{document}